\documentclass[12pt]{amsart}
\usepackage{amssymb}
\usepackage{amsthm} 
\usepackage{amsmath}      
\usepackage{latexsym}
\usepackage{amsfonts}
\usepackage{multicol}
\usepackage{color}
\newtheorem{thm}{Theorem}[section]
\newtheorem{cor}[thm]{Corollary}
\newtheorem{lem}[thm]{Lemma}
\newtheorem{prop}[thm]{Proposition}
\theoremstyle{definition}
\newtheorem{defn}[thm]{Definition}
\newtheorem{rem}[thm]{Remark}

\numberwithin{equation}{section}

\DeclareMathOperator{\ima}{{Im}}
\DeclareMathOperator{\End}{{End}}
\DeclareMathOperator{\Sp}{Sp}

\newcommand{\cW}{\mathcal{W}}

\newcommand{\QQ}{\mathbb Q}
\newcommand{\NN}{\mathbb N}
\newcommand{\ZZ}{\mathbb Z}
\newcommand{\CC}{\mathbb C}

\begin{document}
\bibliographystyle{amsplain}
\title[Dihedral actions on JX]{Decomposition of Jacobian varieties of curves with dihedral actions via equisymmetric stratification.}
\author{Milagros Izquierdo$^1$, Leslie Jim\'enez$^1$ and Anita Rojas$^2$}

\address{1\\Matimatiska Institutionen, Link\"{o}pings Universitet, Link\"{o}ping\\Sweden}
\email{milagros.izquierdo@liu.se, leslie.jimenez@liu.se}
\address{2\\Departamento de Matem\'aticas, Universidad de Chile, Santiago\\Chile}
\email{anirojas@uchile.cl}

\thanks{Partially supported by Fondecyt Grant 1140507, Conicyt PIA ACT1415 and Becas Chile Fellowship for Postdoctoral studies}

\subjclass{Primary 14H40; Secondary 14H30}

\keywords{Group algebra decomposition, Jacobians with group action, compact Riemann surfaces}


\begin{abstract} 
Given a compact Riemann surface $X$ with an action of a finite group $G$, the group algebra $\QQ[G]$ provides an isogenous decomposition of its Jacobian variety $JX$, known as the group algebra decomposition of $JX$. We consider the set of equisymmetric Riemann surfaces $\mathcal{M}(2n-1, D_{2n}, \theta)$ for all $n\geq 2$. We study the group algebra decomposition of the Jacobian $JX$ of every curve $X\in \mathcal{M}(2n-1, D_{2n},\theta)$ for all admissible actions, and we provide affine models for them. We use the topological equivalence of actions on the curves to obtain facts regarding its Jacobians. We describe some of the factors of $JX$ as Jacobian (or Prym) varieties of intermediate coverings. Finally, we compute the dimension of the corresponding Shimura domains.
\end{abstract}

\maketitle

\section{Introduction}

Let $G$ be a finite group acting on a given compact Riemann surface $X$ of genus $g\geq 2$. This action induces a homomorphism $\rho:\QQ[G]\to \text{End}_{\QQ}(JX)$ from the rational group algebra $\QQ[G]$ into the rational endomorphism algebra of the Jacobian variety $JX$ of $X$ in a natural way. The factorization of $\QQ[G]$ into a product of simple algebras yields a decomposition of $JX$ into abelian subvarieties \cite{l-r}, \cite{kr}, called \textit{the group algebra decomposition of the Jacobian variety}.  

Jacobians with group action, and in particular the group algebra decomposition, have been extensively studied from different points of view in \cite{accola1},  \cite{brr},  \cite{cr}, \cite{serre}, \cite{yo}, \cite{kr}, \cite{lro}, \cite{l-r}, \cite{paulhus}, \cite{paulhus2}, \cite{sevin}, \cite{reci-ro}, \cite{rubi}, \cite{yoibero}. Regarding dihedral groups acting on Jacobians,  studies of this can be found in \cite{crro}, from an algebraic point of view, and for particular actions on abelian varieties in \cite{lrro}. 

Since the group algebra decomposition of $JX$ comes from algebraic data of the group $G$ -its irreducible representations \cite{cr, l-r}- this decomposition has been studied mostly, and for many years, from an algebraic point of view. However, further information about the decomposition, such as the dimensions and the polarizations of the factors, depends on the {\it geometry} of the action of $G$ on $X$. For instance, the dimension of the factors is explicitly given in terms of the monodromy of the action \cite{yoibero}. Due to this dependence, it comes the general question about how the choice of the action of $G$ on $X$ affects the group algebra decomposition of $JX$. Very little is known about this question. In this work, we propose an approach through equisymmetric stratification (topological equivalence of actions).  

Consider the equisymmetric stratification $\mathcal{M}_g = \cup \mathcal{M}(g,G,\theta) $ of the moduli space of Riemann surfaces of  genus $g$ defined in \cite{bro}. We use the equisymmetric stratification to obtain facts regarding Jacobian varieties of surfaces.  In particular, we study the group algebra decomposition of the Jacobian variety of a curve $X \in \mathcal{M}(2n - 1, D_{2n}, \theta)$, where $\mathcal{M}(2n - 1, D_{2n}, \theta)$ is the  equisymmetric stratum of  curves of genus $2n - 1$ with a $D_{2n}-$action and epimorphism $\theta: \Delta\to G$ determining an action $\sigma$ of $D_{2n}$ on $X$ with signature $(0; 2, 2, 2, 2, n)$. In other words, if $X, Y$ belong to the same stratum $\mathcal{M}(2n - 1, D_{2n}, \theta)$ then the group algebra decomposition of $JX$ and $JY$ are the same up to a permutation of the factors. On the other hand, if $X$ and $Y$ belong to different strata, then the decompositions are completely different (see section \ref{S:decomposition}). Besides, using the method introduced in \cite{yo} we describe the group algebra decomposition of the corresponding Jacobian varieties in terms of Jacobian (or prym) varieties of intermediate coverings. This result refines what was obtained in \cite{crro} for the action we study.


In section \ref{prelimi} we introduce some preliminary concepts. In section \ref{S:strata}, we study the actions and describe the corresponding equisymmetric strata. In section \ref{S:equations} we find equations describing the curves in these strata as affine plane curves. In sections \ref{S:decomposition} and \ref{S:shimura} we study the corresponding Jacobians and their loci in the moduli space of principally polarized abelian varieties. We describe the group algebra decomposition of the Jacobians, and we compute the dimension of the Shimura domain corresponding to each action.

\vspace{0.5cm}
\noindent{\bf{Acknowledgments}} The second and third authors are very grateful to the Department of Mathematics of Link\"{o}ping University, where this paper was written, for its hospitality and the kindness of all its people during their postdoctoral and research stay respectively. 


\section{Preliminaries}\label{prelimi}

A compact Riemann surface corresponds to a smooth irreducible projective algebraic curve over $\mathbb C$, we say \lq\lq curve\rq\rq along this work referring to a compact Riemann surface. Any curve $X$ of genus $g$ has  associated a principally polarized abelian variety $JX:=H^{1,0}(X,\CC)^*/H_1(X, \ZZ)$, where $H^{1,0}(X,\CC)^*$ is the dual of the complex vector space of holomorphic forms of $X$, and  $H_1(X, \ZZ)$ is the first homology group of $X$. This variety is called \textit{the Jacobian variety of $X$} and has complex dimension $g$. For references about abelian varieties and Jacobians see for instance \cite{bl} and \cite{rubi}.

We define the group of automorphisms $\text{Aut}(X)$ of a Riemann surface $X$ as the analytical automorphism group of $X$. We say that a finite group $G$ acts on $X$ if there is a monomorphism $\sigma: G\to \text{Aut}(X)$.

The quotient $X/G$ (the space of orbits of the action of $G$ on $X$) corresponds to a compact Riemann orbifold, we denote its genus by $\gamma=g(X/G)$. If $X$ is uniformized by a surface Fuchsian group $\Gamma$, i.e. $X=\mathbb{H}/\Gamma$, then $X/G=\mathbb H/\Delta$. The canonical presentation of $\Delta$ is given by 
\begin{equation}\label{delta}
\Delta=\langle \alpha_1, \beta_1, ..., \alpha_{\gamma}, \beta_{\gamma}, x_1, ..., x_r: x_1^{m_1}=...=x_r^{m_r}= \prod_{j=1}^{\gamma}[\alpha_{j}, \beta_{j}]\prod_{i=1}^r\, x_i=1\rangle.
\end{equation}

Moreover, each action $\sigma$ of $G$ on $X$ is determined by an epimorphism $\theta: \Delta\to G$ from the Fuchsian group $\Delta$ such that $\text{ker}(\theta)=\Gamma$.

Considering the presentation of $\Delta$ as in \eqref{delta}, we define \textit{the signature of G on X} as the vector of numbers $(\gamma; m_1,...,m_r)$.

If we mark every $m_i$ with the conjugacy class of subgroups $C_i$ of $G_i=\langle \theta(x_i)\rangle$ of $G$, we define the \textit{geometric signature of G on X} as $(\gamma; [m_1, C_1],...,[m_r, C_r])$ \cite{yoibero}. The genus of $X$ is given by the Riemann-Hurwitz formula 

$$g(X)=|G|(\gamma-1)+1+\frac{|G|}{2}\sum_{i=1}^{r}\big(1-\frac{1}{m_i}\big).$$

A $2\gamma+r$ tuple $(a_1, \dots, a_{\gamma},b_1, \dots,
b_{\gamma},c_1, \dots,c_r)$ of elements of $G$ is called a
\textit{generating vector of type $(\gamma;m_1,...,m_r)$} if the
following conditions are satisfied: 

\begin{enumerate}
\item $G$ is generated by the elements $(a_1, \dots, a_{\gamma},b_1,
\dots, b_{\gamma},c_1, \dots,c_r)$; 
\item $\text{order}(c_i)=m_i$; and 
\item $\prod_{i=1}^{\gamma}[a_i,b_i]\prod_{j=1}^rc_j=1$, where $[a_i,b_i]$ is the commutator of $a_i, b_i \in G$. 
\end{enumerate}

The existence of a generating vector of a given type ensures the existence of a Riemann surface with an action of the corresponding group with the given signature (Riemann's existence theorem). For a good account on Fuchsian groups and group actions see \cite{bro, macbeath, macbeath-singerman, yoibero, singerman}.

We say that two actions $\sigma_1, \sigma_2$ are {\it topologically equivalent} if there is an $\omega\in \text{Aut}(G)$ and an $h\in \text{Hom}^+(X)$ such that
\begin{equation}\label{equivalentactions}
\sigma_2(g) = h \sigma_1(\omega(g)) h^{-1}
\end{equation}
for all $g\in G$.

Equivalently, two epimorphisms $\theta_1, \theta_2:\Delta\to G$ define two topologically equivalent actions $\sigma_1, \sigma_2$ of $G$ on $X$ if there exist two automorphisms, namely $\phi:\Delta\to \Delta$ and $\omega: G\to G$ such that 
\begin{equation}\label{equivalentactions-fuchsian}
\theta_2=\omega \circ \theta_1 \circ \phi^{-1}
\end{equation}

Let $\mathcal{B}$ be the subgroup of $\text{Aut}(\Delta)$ induced by orientation preserving homeomorphisms. In other words, $\sigma_1, \sigma_2$ are topologically equivalent if and only if the epimorphisms $\theta_1, \theta_2$ lie in the same $\mathcal{B}\times \text{Aut}(G)-$class. For references see \cite{bro, harvey}.

Since we are interested in actions with quotient of genus $0$, we consider only the elements of $\mathcal{B}$ corresponding to compositions of automorphims $\Phi_{i,i+1}\in \text{Aut}(\Delta)$ such that $\Phi_{i,i+1}(x_i)=x_ix_{i+1}x_i^{-1}, \Phi_{i,i+1}(x_{i+1})=x_i, \Phi_{i,i+1}(x_j)=x_j$ for all $j\neq i, i+1.$ 


Let $\mathcal{M}(g, G, \theta)$ be the stratum of surfaces of genus $g$ with automorphism group $G$ in the conjugacy class in the mapping class group of the action determined by the epimorphism $\theta$. Let $\overline{\mathcal{M}}(g, G, \theta)$ be the subset of surfaces having an automorphism group containing the automorphism group $G$ in the conjugacy class determined by $\theta$ in the mapping class group $\text{Mod}(\Gamma)$. 



It is known that if the signature of the action is $(\gamma; m_1,...,m_r)$, then the Teichmuller dimension is $3(\gamma -1) + r$. For references about Teichm\"uller and moduli spaces see \cite{bro, harvey, macbeath-singerman, nag}. 


Let $G$ be a finite group. If $\psi: G\to \text{GL}(V)$ is an irreducible representation of $G$ over $\CC$ afforded by the complex vector space $V$, we denote by $F$ its field of definition and by $K$ the field obtained by extending $\QQ$ by the values of its character $\chi_V$; then $K\subseteq F$ and $m_V=[F:K]$ is the Schur index of $V$. 

If $H$ is a subgroup of $G$, $\text{Ind}_H^G1$ denotes the representation of $G$ induced by the trivial representation of $H$. Besides $\langle U,V \rangle$ denotes the usual inner product of the characters. By Frobenius Reciprocity $\langle \text{Ind}_H^G 1,V\rangle=\text{dim}_\CC \text{Fix}_HV$, where $\text{Fix}_HV$ is the subspace of $V$ fixed under $H$. We use indistinctly $V$ or $\chi_{\psi}$ to denote the representation or the vector space, affording it, when the context is clear. 

A good account about representation theory of finite groups may be found in \cite{curtis, james-liebeck, serre1}.  

Given a compact Riemann surface $X$ with an action of a group $G$ we consider the induced homomorphism  $\rho:\QQ[G]\to \text{End}_\QQ(JX)$. 

$\QQ[G]$ decomposes into a product $Q_0\times \dots \times Q_t$ of simple $\QQ-$algebras, the simple algebras $Q_i$ are in bijective correspondence with the rational irreducible representations of $G$. That is, for any rational irreducible representation $\cW_i$ of $G$ there is a uniquely determined central idempotent $e_i$ generating $Q_i$. Moreover, the decomposition of every $Q_i=L_1\times \dots \times L_{n_i}$ into a product of minimal left ideals (all isomorphic) gives a further decomposition of the Jacobian. There are idempotents $f_{i1},\dots, f_{in_i}\in Q_i$ such that $e_i=f_{i1}+\dots +f_{in_i}$ where $n_i=\text{dim} V_i/m_{V_i}$, with $V_i$ a $\CC-$irreducible representation associated to $\cW_i$. The factor $ B_{i}^{n_i}$ is defined as the image of $m\rho(e_i)\in \text{End}(JX)$, and it is shown that it does not depend on $m$ (up to isogeny). In the same way the factor $B_i$ corresponds to the image of any of the $f_{in_i}$. See \cite{bl, l-r} for details regarding these decompositions. 

It is known that the factor associated to the trivial representation of $G$, i.e. the image of $e_0\in Q_0$, is isogenous to $J_G=J(X/G)$. The addition map is an isogeny \cite{bl, cr, l-r}. 
\begin{equation}\label{nu-tilde}
\nu:J_G\times B_{1}^{n_1}\times \dots \times B_{t}^{n_t}\to JX.
\end{equation}
This is called \textit{the group algebra decomposition of $JX$}. We denote the isogeny $\nu$ by $JX\sim J_G\times B_{1}^{n_1}\times \dots \times B_{t}^{n_t}$.

\subsection{Regarding the factors}\label{SS:factors}
Consider the decomposition of $JX$ given in \eqref{nu-tilde} and $H\leq G$, then the corresponding group algebra decomposition of the Jacobian variety $J_H=J(X/H)$ of the intermediate quotient $X/H$  is given as follows 
\begin{equation}\label{carocca-rodriguez}
J_H \sim J_G\times B_{1}^{\frac{\text{dim} \text{Fix}_HV_1}{m_1}}\times \dots \times B_{t}^{\frac{\text{dim} \text{Fix}_HV_t}{m_t}}
\end{equation}

As in \cite{cr}, for every $H\leq G$ define $p_H=\frac{1}{\vert H\vert}\sum_{h\in H} h$. It is a central idempotent in $\QQ[H]$ corresponding to the trivial representation of $H$.
Moreover, $\ima (p_H)=\pi_H^*(J_H)$, where $\pi_H^*(J_H)$ is the pullback of $J_H$ by $\pi_H:X\to X/H$. Finally define $f_H^{i}=p_H e_{i}$,  which is an idempotent in $\QQ[G]e_{i}$, if $\text{dim} \text{Fix}_HV_i\neq 0$ then
\begin{equation}\label{fh}
\ima (f_H^i)=B_{i}^{\frac{\text{dim} \text{Fix}_HV_i}{m_{i}}}.
\end{equation}

The induced action of $G$ on $JX$ provides geometrical information about the components of the group algebra decomposition of $JX$ \cite{cr}. In fact, the dimension of the subvarieties in the decomposition (\ref{nu-tilde}) are obtained using the generating vector of the action \cite[Theorem 5.2]{yoibero}.

\begin{thm}[Dimension of the $B_j$'s]\label{dimensiones}
Let $G$ be a finite group acting on a compact Riemann surface $X$ with geometric signature given by $(\gamma; [m_1,C_1], ..., [m_r, C_r])$. Then the dimension of any subvariety $B_i$ associated to a non trivial rational irreducible representation $\cW_i$ in (\ref{nu-tilde}), is given by
$$\dim B_i=k_i\Big(\dim V_i(\gamma -1)+\frac{1}{2}\sum_{k=1}^r (\dim V_i-\dim \text{Fix}_{G_k} V_i)\Big),$$
where $G_k$ is a representative of the conjugacy class $C_k$, $\dim V_i$ is the dimension of a complex irreducible representation $V_i$ associated to $\cW_i$, $K_i = K_{V_i}$ and $k_i=m_i [K_i:\QQ]$.  
\end{thm}

To identify the factors as Jacobians of intermediate coverings we use the following  result \cite[Lemma 1]{yo}.
\begin{lem}\label{metodo}
Let $X$ be a Riemann surface with an action of a finite group $G$ such that the genus of $X/G$ is equal to zero. Consider $\nu$ the group algebra decomposition of $JX$ as in \ref{nu-tilde}.  
\begin{itemize}
\item[(i)] If $H\leq G$ is such that $\dim_\CC \text{Fix}_HV_i=m_i$, where $m_i$ is the Schur index of the representation $V_i$, then we have that $\ima (f_H^i)=B_{i}.$ In addition, 
\item[(ii)] if $\dim_\CC \text{Fix}_HV_l = 0$ for all $l$, $l\neq i$, 
such that $\dim_\CC B_l\neq 0$ in the isotypical decomposition of $JX$ in Equation (\ref{nu-tilde}), then $$J_H\sim \ima (p_H)=B_{i}.$$
\end{itemize}
\end{lem}

\subsection{Rational irreducible representations of the dihedral group}\label{rationalrep}

Let $D_{2n}=\langle a,s: a^{2n}=s^2=(as)^2=1\rangle$ be the dihedral group of order $4n$. It is known (see for instance \cite{crro, james-liebeck}) that $D_{2n}$ has four complex irreducible representations of degree one denoted by $\chi_0, \chi_1, \chi_2, \chi_3$, and $n-1$ of degree two $\psi_j$ for $1\leq j\leq n-1$ given by 

$$\psi_j (a)= \left(\begin{array}{cc} \omega^j & 0 \\
0 & \omega^{-j}\end{array}\right)\text{ and }\psi_j(s)= \left(\begin{array}{cc} 0 & 1 \\
1 & 0\end{array}\right),$$
where $\omega=\omega_{2n}=\text{exp}\frac{\pi i}{n}$. The character table is in Table 1

\vskip12pt

\begin{center}\label{TbD}
\begin{tabular}{|l|c|c|c|c|c|c|r|}\hline
Conjugate classes in $D_{2n}$: &  1 & $a^n $ & $a^r (1\leq r\leq n-1)$ & $s$ & $as$ \\ \hline
$\chi_0$   &  1 & 1 & 1 & 1 & 1 \\
$\chi_1$   &  1 & 1 & 1 & -1 & -1 \\
$\chi_2$   &  1 & $(-1)^n$ & $(-1)^r$ & 1 &  -1 \\
$\chi_3$   &  1 & $(-1)^n$ & $(-1)^r$ & -1 & 1 \\
$\chi_{\psi_j}$   &  2 & $2(-1)^j$ & $\omega^{jr} + \omega^{-jr}$ & 0 & 0 \\ \hline
\end{tabular}
\vspace{0.3cm}
\\
\textsc{Table 1.} Character Table of $D_{2n}$.
\end{center}
\vspace{0.3cm}

We are interested in the rational irreducible representations of $D_{2n}$. Let 
$$\Omega(2n)=\{d:d\text{ is a divisor of }2n\text{ and } d<n\}.$$


For each $d$ in $\Omega(2n)$, consider $\psi_d$ as above. Its field of definition coincides with the field generated by its characters. It means that the Schur index $m_d=1$ and  
$$K_{\psi_d}=\QQ(\omega_{2n}^d + \omega_{2n}^{-d})=\QQ(\omega_{\frac{2n}{d}}+ \omega^{-1}_{\frac{2n}{d}}).$$
Furthermore $[K_{\psi_d}:\QQ]=\frac{\varphi(\frac{2n}{d})}{2},$ where $\varphi$ denotes the Euler function. Hence, the rational irreducible representation $W_{\psi_d}$ of $D_{2n}$ satisfies 
$$K_{\psi_d}\otimes_{\QQ}W_{\psi_d}= \oplus_{\sigma\in \text{Gal}(K_{\psi_d}/\QQ)} \psi_d^{\sigma}$$
and it is of  degree $\text{deg}(W_{\psi_d})=\text{dim}\psi_d [K_{\psi_d}:\QQ]=2\frac{\varphi(\frac{2n}{d})}{2}=\varphi\displaystyle{(\frac{2n}{d})}$.

\begin{rem}\label{R:outer}

Observing Table 1 of the complex characters of $D_{2n}$, we notice that the characters $\chi_2$ and $\chi_3$ have the same values for each conjugation class except for the classes $\mathcal{C}_1=\{\langle a^{2l+1}s\rangle:0\leq l \leq n-1\}$ and $\mathcal{C}_2=\{\langle a^{2l}s\rangle:0\leq l \leq n-1\}$ represented by the elements $as$ and $s$ respectively.
Hence, in particular we have $\chi_2(s)=\chi_3(as)$. In fact, we have the following results

\begin{itemize}
\item {\bf Result 1.}
$\chi_2=\chi_3\circ \omega$ where $\omega\in \text{Out}(D_{2n})\subset \text{Aut}(D_{2n})$ is the outer automorphism given by $\omega(a)=a, \omega(s)=as$. 

\item {\bf Result 2.} Let $\mathcal{C}_1, \mathcal{C}_2$ be the conjugation classes defined above, and $\omega$ the automorphism in Result 1, then $\text{Fix}_{H}\chi_2=\text{Fix}_{\omega(H)}\chi_3$ for all $H\in \mathcal{C}_2$, where $\text{Fix}_{H}\chi$ is the fixed space defined in Section \ref{prelimi}. In particular, $\text{Fix}_{\langle s\rangle}\chi_2=\text{Fix}_{\langle\omega (s)\rangle}\chi_3$.
\end{itemize}
\end{rem}


\section{On actions of $D_{2n}$ with signature $(0;2,2,2,2,n)$.}\label{S:strata}

We recall for this section that two actions $\sigma_1, \sigma_2$ of $G$ on $X$ are topologically equivalent if they satisfy equation (\ref{equivalentactions}) i.e. if the epimorphisms $\theta_1, \theta_2:\Delta\to G$, lie in the same $\mathcal{B}\times \text{Aut}(G)-$class as explained in the Section \ref{prelimi}. 

In this section, we classify the different generating vectors of type $(0;2,2,2,2,n)$, i.e. we classify all the $5-$tuples $(c_1, \dots,c_5)$ of elements of $D_{2n}$ such that $D_{2n}$ is generated by $c_1, \dots,c_5$ with $\text{order}(c_i)=2, 1\leq i\leq 4$, $\text{order}(c_5)=n$, and $\prod_{i=1}^5c_i=1$. The results depend on the parity of $n$, thus we treat both cases separately. 

\subsection{$n$ an odd number.}

\begin{prop}\label{vectoresgeneradores}
Let $n\geq 3$ be an odd number. Then, every generating vector of $D_{2n}$ of type $(0;2,2,2,2,n)$ belongs to one of the following classes with representatives 
\begin{itemize}
\item[I.] $(a^n, a^n, s, a^2 s, a^2)$,
\item[II.] $(s, s, as, a^3 s, a^2)$.
\end{itemize}
\end{prop}

\begin{proof}

Consider an action $\sigma$ of $D_{2n} = \langle a, s | a^{2n}=s^2 = (as)^2 = 1 \rangle $ with monodromy $\theta: \Delta \to D_{2n}, s(\Delta)= (0; 2,2,2,2,n)$, with generating vector $(\theta(x_1), \theta(x_2), \theta(x_3), \theta(x_4), \theta(x_5))$. First of all we can assume, up to an automorphism of $D_{2n}$, that $\theta(x_5) = a^2$.
Now as $\theta$ is an epimorphism the condition $\theta(x_1x_2x_3x_4x_5) = 1$ obliges to one of the following cases:

\begin{itemize}
\item[I.] Two of $x_1, x_2, x_3, x_4$ are mapped to $a^n$ and the other two to $a^{t_1}s$ and $a^{t_2}s$ such that $t_1-t_2 \equiv 2 \mod(2n)$, or

\item[II.] All $x_1, x_2, x_3, x_4$ are mapped to $a^{t_{i}}s$, $1\le i \le 4$, such that $t_1-t_2+t_3-t_4 \equiv 2 \mod(2n)$.
\end{itemize}

 Observe that in case I the parity of $t_1$ and $t_2$ coincides. Applying a suitable composition of automorphisms of $\Delta$ we obtain that a generating vector of an action $\sigma$ in this case is in principle of the following types

 i) $(a^n, a^n, a^{t_1}s, a^{t_2}s, a^2)$, with $t_1$ and $t_2$ even,

 ii) $(a^n, a^n, a^{t_1}s, a^{t_2}s, a^2)$, with $t_1$ and $t_2$ odd,

For instance $\Phi_{2,3}\in \mathcal{B}$ conjugates the generating vector $(a^n, s, a^n, a^2s, a^2)$ with $(a^n, a^n, s, a^2s, a^2)$, since $ \Phi_{2,3}(x_2) =x_2x_3x_2^{-1},   \Phi_{2,3}(x_3) =x_2$ and $\Phi_{2,3}(x_i) =x_i $ for $i\neq 2,3$.


Finally applying an inner automorphism of $D_{2n}$ one gets that any generating vector of type i) is equivalent to $(a^n, a^n, s, a^2s, a^2)$ and any generating vector of type ii) is equivalent to $(a^n, a^n, as, a^3s, a^2)$. Using an outer automorphism of $D_{2n}$, such as $s\mapsto as$, one gets that any generating vector in this case is equivalent to $(a^n, a^n, s, a^2s, a^2)$. 

In case II at least one of $t_{i}$ must be odd and one even, and the parity of $t_1+t_3$ equals the one of $t_2+t_4$, these facts force the exponents $t_i$ to be two odd and two even. This is, 

 $(a^{t_1}s, a^{t_2}s, a^{t_3}s, a^{t_4}s, a^2)$, with two t$_{i}$'s  even numbers, and two $t_i$'s odd. 

We proceed in an analogous way as in case I. For instance to prove that  $(s,s,as,a^3s,a^2)$ is equivalent to $(as,as,s,a^2s,a^2)$, hence they correspond to the same action (up to topological equivalence), we use the element $\Phi_{2,3}^{-1}\Phi_{1,2}^{-1}\Phi_{3,4}\Phi_{2,3}^{-1}\in \mathcal{B}$. The element $\Phi_{2,3}^{-1}\Phi_{1,2}^{-1}\Phi_{3,4}\Phi_{1,2}\Phi_{2,3}^{-1}\in \mathcal{B}$ conjugates $(as,as,a^2s,a^4s,a^2)$ to $(as,as,s,a^2s,a^2)$, and the element $\Phi_{2,3}\Phi_{1,2}^{-1}\Phi_{3,4}^{-1}\Phi_{2,3}^{-1}\in \mathcal{B}$  $(s,s,as,a^3s,a^2)$ to $(as,as,a^2s,a^4s,a^2)$. 

Finally, we obtain that any generating vector in this case is equivalent to $(s, s, as, a^3s, a^2)$. 

To prove that $\sigma_1$ is not equivalent to $ \sigma_2$, it is enough to note that $\langle a^n \rangle$ is a characteristic subgroup of $D_{2n}$.

\end{proof}

\begin{defn}\label{D:sigmas}
Let $n$ be an odd number. Let $\sigma$ be an action of $D_{2n}$ on a curve with signature $(0;2,2,2,2,n)$. We say that $\sigma$ is an action of
\begin{itemize}
\item[(i)] Type 1, we write $\sigma_1$, if its generating vector is equivalent to $(a^n, a^n, s, a^2 s, a^2)$. 
\item[(ii)] Type 2, we write $\sigma_2$, if its generating vector is equivalent to $(s, s, as, a^3 s, a^2)$.
\end{itemize}
\end{defn}

A direct consequence from Proposition \ref{vectoresgeneradores} is that there are two non equivalent topological actions in correspondence with the two classes of generating vectors. Thus, there are two conjugacy classes of $D_{2n}$ in the mapping class group $\text{Mod}(\Gamma_{2n-1})$ for $n$ odd.








We summarize all the results in this subsection in Theorem \ref{T:strata}. 

\begin{thm}\label{T:strata}
The family of compact Riemann surfaces with action of $D_{2n}$ of type $(0;2,2,2,2,n)$ has two equisymmetric strata each one of dimension $2$: $\mathcal{M}(2n-1, D_{2n}, \theta_1)$ and $\mathcal{M}(2n-1, D_{2n}, \theta_2)$, where $\theta_1, \theta_2$ are the epimorphisms corresponding to the actions $\sigma_1, \sigma_2$ in Definition \ref{D:sigmas}. 
\end{thm}

\subsection{$n$ an even number.}

\begin{prop}\label{vectoresgeneradores_2}

Let $n\geq 2$ be an even number. Then, every generating vector of $D_{2n}$ of type $(0;2,2,2,2,n)$ belongs to the class represented by $(s, s, as, a^3 s, a^2)$.  Hence there is just one action $\sigma$ up to topological equivalence.
\end{prop}

\begin{proof}
We proceed as in Proposition \ref{vectoresgeneradores}, except that here we obtain just one equivalence class of generating vectors; the vector $ (a^n, a^n, as, a^3s, a^2)$ does not generate $D_{2n}$ for even $n$ since the center of $D_{2n}$ is cyclic of order $2$ and the center of $D_n\times \mathbb{Z}_2$ is isomorphic to the Klein group. The last statement follows from the fact that topologically non equivalent actions are in correspondence with orbits of generating vectors.
\end{proof}





We summarize the results in this subsection in the following theorem.

\begin{thm}
Let $n$ be an even number. Then the family of compact Riemann surfaces with action of $D_{2n}$ of type $(0;2,2,2,2,n)$ has one equisymmetric stratum of dimension $2$: $\mathcal{M}(2n-1, D_{2n}, \theta)$ where $\theta$ is the epimorphism corresponding to the action $\sigma$ in Proposition \ref{vectoresgeneradores_2}. 
\end{thm}


\section{Equation for the curves.}\label{S:equations}

We want to find equations for the curves $X$ with the action(s) of $D_{2n}$ with signature $(0;2,2,2,2,n)$. From the theory developed before (Section \ref{S:strata}), we know that the actions correspond to two strata in the case $n$ odd with generating vectors $(a^n, a^n, s,a^2s,a^2)$ and $(s,s, as,a^3s,a^2)$, for the actions labeled $\sigma_1$ and $\sigma_2$ respectively (see Definition \ref{D:sigmas}), and one stratum for $n$ even with generating vector $(s,s, as,a^3s,a^2)$ (see Proposition \ref{vectoresgeneradores_2}). We recall that each action determines a $2-$dimensional equisymmetric stratum, hence our planar affine models will depend on two complex parameters.

\begin{thm}
Let $G$ be the dihedral group $D_{2n}=\langle a^{2n}=s^2=(as)^2=1\rangle$ acting on genus $g=2n-1$ with signature $(0;2,2,2,2,n)$. If $n$ is odd, denote by $\mathcal{M}(g,G,\theta_1)$ and $\mathcal{M}(g,G,\theta_2)$ the two $2-$dimensional equisymmetric strata corresponding to the actions $\sigma_1$ and $\sigma_2$. If $n$ is even, denote by $\mathcal{M}(g,G,\theta_2)$ the $2-$dimensional equisymmetric stratum corresponding to the (unique) action $\sigma_2$. Then

\begin{enumerate}
\item The Riemann surfaces in $\mathcal{M}(g,G,\theta_1)$ are hyperelliptic curves, and an affine plane model representing them is given by

$$\mathcal{H}_{\lambda,\mu}: y^2=(x^n-\lambda^n)(x^n-\frac{1}{\lambda^n})(x^n-\mu^n)(x^n-\frac{1}{\mu^n}),$$
with $\lambda,\mu \in \mathbb{C}\setminus\{0\}$ and $\lambda^{2n}, \mu^{2n}\neq 1$.

\item The Riemann surfaces in $\mathcal{M}(g,G,\theta_2)$ are elliptic $n-$gonal curves. An affine plane model representing them is given by

$$\mathcal{E}_{a,b}: x^{2n}+y^{2n}+ax^ny^n+bx^n+by^n+1=0,$$
with $a,b\in \mathbb{C}\setminus \{0\}$.
\end{enumerate}
\end{thm}
\begin{proof}

Let $X$ be a Riemann surface in the strata $\mathcal{M}(g,G,\theta_1)$, and consider the central subgroup $H=\langle a^n\rangle\leq D_{2n}$ of order $2$. Since $n$ is odd in this case, we have $G=D_{2n}= D_n\times H$. A generating vector for the actions in this strata is $(a^n,a^n,s,a^2s,a^2)$, then we have an epimorphism $\theta_1:\Delta \to G$ defining the action $\sigma_1$. Let $\Lambda=\theta_1^{-1}(H)$. Since $H$ is a normal subgroup of $G$, hence $N=\Delta/\Lambda\cong D_{2n}/H \cong D_n$. Following \cite{macbeath, singerman}, we compute the ramification of the covering $X\to X/H$. This is, for every generating element $x_i$ in the canonical presentation \eqref{delta} of $\Delta$ we determine the order $n_i$ of the class $\overline{x_i}$ in $N$, thus in the quotient it contributes with $2n/n_i$ points with branching number $m_i/n_i$, where $m_i$ is the order of $x_i$ in $\Delta$. For the action $\sigma_1$ we have

\begin{enumerate}
\item order$(\overline{x_1})=$ order$(\overline{x_2})=$ order$(\overline{a^n})=1$ in $N$. Therefore each one contributes with $2n$ points of branching number $2$.
\item order$(\overline{x_3})=$ order$(\overline{s})=$ order$(\overline{a^2s})=$ order$(\overline{x_4})=2$ in $N$. Therefore each one contributes with $n$ regular points.
\item order$(\overline{x_5})=$ order$(\overline{a^2})=n$ in $N$. Therefore it contributes with two regular points.
\end{enumerate}

Summarizing, the covering $f:X\to Y=X/H$ has $4n$ branch points.  By Riemann-Hurwitz we obtain that the quotient surface $Y=X/H$ is of genus $0$, hence $X$ is hyperelliptic and the group $G/H \cong D_n$ acts on the branch points having two orbits of $2n$ points each.


We apply the results in \cite{w1} or \cite{jaime}, take one representative $\lambda$ and $\mu$ for each orbit and set $w=e^{2\pi i/n}$. Then the full orbit of each point is
$\{\lambda, w\lambda, w^2\lambda,\dots, w^{n-1}\lambda, w\lambda^{-1}, w^2\lambda^{-1},\dots, w^{n-1}\lambda^{-1}\}$ and analogously $\{\mu, w\mu, w^2\mu,\dots, w^{n-1}\mu, w\mu^{-1}, w^2\mu^{-1},\dots, w^{n-1}\mu^{-1}\}.$

Therefore one equation for $X$ is

$$y^2=\prod_{k=0}^{n-1}(x-w^k\lambda)\prod_{k=0}^{n-1}(x-\frac{w^k}{\lambda})\prod_{k=0}^{n-1}(x-w^k\mu)\prod_{k=0}^{n-1}(x-\frac{w^k}{\mu}),$$
the result follows from this.

For surfaces given by $\theta_2:\Delta\to G$, i.e. with action $\sigma_2$ and generating vector $(s,s,as,a^3s,a^2)$, consider the subgroups $H=\langle a^2\rangle$ and $\Lambda=\theta_2^{-1}(H)$ of index $4$ in $G$ respectively $\Delta$. As before, let $N=\Delta/\Lambda\cong D_{2n}/H$. Again by \cite{macbeath, singerman} we have

\begin{enumerate}
\item order$(\overline{x_1})=$ order$(\overline{x_2})=$ order$(\overline{s})=2$ in $N$. Therefore each one contributes with regular points.
\item order$(\overline{x_3})=$ order$(\overline{as})=$ order$(\overline{a^3s})=$ order$(\overline{x_4})=2$ in $N$. Therefore each one contributes with two regular points.
\item order$(\overline{x_5})=$ order$(\overline{a^2})=1$ in $N$. Therefore it contributes with $4$ conic points with branch number $p$.
\end{enumerate}

Using Riemann-Hurwitz to compute the genus $\gamma$ of the quotient $Y=X/H$ we get $2n-1=n(\gamma-1)+1+\frac{1}{2}(4(n-1),$
then $\gamma=1$, hence $X$ is a so called elliptic $n-$gonal curve. An affine plane model for this kind of curves is

$$x^{2n}+y^{2n}+ax^ny^n+bx^n+by^n+1=0,$$
with $a,b \in \mathbb{C}\setminus\{0\}.$


\end{proof}

\begin{rem}

\begin{enumerate}
\item Notice that for $n=2$ we only have the stratum determined by $\sigma_2$, and the corresponding Riemann surfaces are of genus $3$. This locus corresponds to the one described in the second row of \cite[Table 2]{magaard}. 

\item Notice that $\mathcal{E}_{0,0} \in \overline{\mathcal{M}}(g,G,\theta_2)$ has a larger automorphism group. In fact $\mathcal{E}_{0,0}$ is the Fermat curve $F_{2n}: x^{2n}+y^{2n}+1=0$, which has $\text{Aut}(F_{2n})=(\mathbb{Z}_{2n}\times \mathbb{Z}_{2n})\rtimes S_3$.
\end{enumerate}
\end{rem}

\section{Group algebra decomposition of Jacobians of curves with $D_{2n}$-action and signature $(0;2,2,2,2,n)$}\label{S:decomposition}

Sections \ref{S:strata} and \ref{S:equations} were devoted to study curves with our actions of $D_{2p}$, and their strata in the moduli space of Riemann surfaces. In the following sections, \ref{S:decomposition} and \ref{S:shimura}, we develop results concerning the corresponding Jacobian varieties and their Shimura domains in the Siegel upper half space.
In particular, in this section we study the group algebra decomposition of the Jacobians associated to these actions. 

\subsection{$n$ an odd number}\label{section-odd}

Let us consider as before $D_{2n}=\langle a,s: a^{2n}=s^2=(as)^2=1\rangle$ acting on a curve $X$ of genus $g=2n-1$ with signature $(0;2,2,2,2,n)$. From Section \ref{prelimi}, we know that  $D_{2n}$ has four complex irreducible representations of degree one and $n-1$ of degree two. All of them with Schur index equal to 1. 

To make further calculations easier, we define 
\begin{equation}\label{eq:omega_odd}
\Omega(2n)=\Omega_{\text{odd}} \cup \Omega_{\text{even}},
\end{equation}
with $\Omega_{\text{odd}}=\{d \in \Omega(2n): d \text{ is an odd number}\}$, $\Omega_{\text{even}}=\Omega(2n)\setminus \Omega_{\text{odd}}$.



The rational irreducible representations of $D_{2n}$ are $\chi_0,\chi_1, \chi_2,\chi_3$ of degree one, and  
$$\cW_{d}:=\cW_{\psi_d}=\bigoplus_{\sigma\in \text{Gal}(K_{\psi_{d}}/\QQ)} \psi_{d}^{\sigma},$$
for all $d\in \Omega(2n)$. 

Then the group algebra decomposition of $JX$ is given by
$$JX\sim E_0\times E_1\times E_2\times E_3\times \prod_{d\in \Omega{\text{odd}}} B_d^2\times \prod_{d\in \Omega{\text{even}}} B_d^2,$$ 
where $E_i$ is a subvariety of $JX$ associated to $\chi_i$, and $B_0=J(X/D_{2n})$ has dimension $0$ in our case.

To compute the dimensions of these subvarieties we use Theorem \ref{dimensiones}; in this case we get

\begin{equation}\label{dim-casodiedral-2}
\dim E_i=-\dim \chi_i+\frac{1}{2}\sum_{k=1}^5 (\dim \chi_i-\dim \text{Fix}_{G_k} \chi_i),
\end{equation}
\begin{equation*}
\dim B_d=\frac{1}{2}\varphi(\frac{2n}{d})\Big(-\dim \psi_d+\frac{1}{2}\sum_{k=1}^5 (\dim \psi_d-\dim \text{Fix}_{G_k} \psi_d)\Big)
\end{equation*}
where $0\leq i\leq 3$ and $d\in \Omega(2n)$.

Therefore we need the dimensions of $\text{Fix}_{G_k} \chi_i$ and $\text{Fix}_{G_k} V_d$, which are obtained in Proposition \ref{propfix}.  

\begin{prop}\label{propfix}
The dimension of the spaces $\text{Fix}_{H} \chi_i$ and $\text{Fix}_{G_k} V_d$ for every $H$ non-trivial cyclic subgroup of $G$ is given in Table 2.

\vspace{0.3cm}


\center{\rm
\begin{tabular}{|l|c|c|c|c|c|c|c|r|}\hline

Subgroups, $r\in \NN$ & $\chi_0$ & $\chi_1$ & $\chi_2$ & $\chi_3$ & $\psi_{\{d: d\in\Omega_{\text{even}}\}}$ & $\psi_{\{d: d\in\Omega_{\text{odd}}\}}$ \\ \hline
$\langle a^n\rangle$  & 1 & 1 & 0 & 0 & 2 & 0 \\
$\langle a^{2r}\rangle$  & 1 & 1 & 1 & 1 & 0 & 0 \\
$\langle a^{2r+1}\rangle$  & 1 & 1 & 0 & 0 & 0 & 0 \\
$\langle s\rangle$   & 1 & 0 & 1 & 0 & 1 & 1\\
$\langle a^{2r+1} s\rangle$   & 1 & 0 & 0& 1 & 1 & 1\\
$\langle a^{2r} s\rangle$   & 1 & 0 & 1& 0 & 1 & 1\\ \hline
\end{tabular}

\vskip12pt

\textsc{Table 2.} Dimension of fixed spaces}

\end{prop} 

\begin{proof}
We want to calculate the dimension of $\text{Fix}_H V$ for $V$ a complex irreducible representation associated to $\chi_i$ and $\cW_j$, for all $i$ and all $j$, and all cyclic subgroup $H$ of $G$. Observe that for the representations of degree 1; $\chi_0, \chi_1, \chi_2$ and $\chi_3$, we only need to check in the character table written before if it is equal to $1$ or $-1$. Then, if $H=\langle h\rangle$
$$\text{dim }\text{Fix}_{H} \chi_i=
\left\{
	\begin{array}{ll}
		1  & \mbox{if } \chi_i (h) = 1 \\
		0 & \mbox{if } \chi_i (h) = -1
	\end{array}
\right.$$
for all $0\leq i\leq 3$. Hence, we obtain the first fourth columns of Table 2. 

Now, for the representations of degree 2, we have 
$$\psi_d (a^n)= \left(\begin{array}{cc} \omega^{nj} & 0 \\
0 & \omega^{-nd}\end{array}\right)= \left(\begin{array}{cc} (-1)^d & 0 \\
0 & (-1)^d\end{array}\right),$$

Then when $d$ is odd, the matrix $\psi_d (a^n)=-I$ and the unique eigenvalue is $\lambda=-1$, then $\text{dim }\text{Fix}_{\langle a^n\rangle} \psi_d= 0$. Now, on the other hand when $d$ is even, the matrix $\psi_d (a^n)=I$ and the unique eigenvalue is $\lambda=1$, then $\text{dim }\text{Fix}_{\langle a^n\rangle} \psi_d= 2$. 

Now for computing the dimension of the fixed spaces by $\langle a^r\rangle$, with $r\neq 0, n$, consider
$$\psi_d (a^r)= \left(\begin{array}{cc} \omega^{rd} & 0 \\
0 & \omega^{-rd}\end{array}\right),$$
for all $d$. Then in any case the eigenvalues are $\lambda=\omega^{rd}, \omega^{-rd}$, then $\text{dim }\text{Fix}_{\langle a^r\rangle} \psi_{d}= 0$. 

To compute the corresponding dimension for $\langle s\rangle$ consider
$$\psi_d (s)= \left(\begin{array}{cc} 0 & 1 \\
1 & 0\end{array}\right).$$

Then the eigenvalues are $\lambda=1, -1$, hence $\text{dim }\text{Fix}_{\langle s\rangle} \psi_d= 1$ for all $d$. Using the same methods for $\langle a^{2r} s\rangle$ and $\langle a^{2r+1} s\rangle$, we complete the table. 
\end{proof}

To know whether a Jacobian variety has elliptic factors has been deeply studied  \cite{serre, kr, paulhus, paulhus2}. We prove in the following theorem that the Jacobian of a curve with action of $D_{2n}$ and signature $(0;2,2,2,2,n)$ has always {\it at least} one elliptic factor in its group algebra decomposition. 

\begin{thm}\label{dimensionesfactores}
If $X$ is a Riemann surface with action of the group $D_{2n}$ and signature $(0;2,2,2,2,n)$ as in Theorem \ref{T:strata}, then $JX$ has always (at least) one elliptic factor in its group algebra decomposition. In fact, the decompositions are 
\begin{itemize} 
\item[(i)] for the action $\sigma_1$, $JX \sim E_2\times \prod_{\{d \in \Omega_{\text{odd}}\}}B_{d}^2$
\item[(ii)] for $\sigma_2$, $JX  \sim E_1\times \prod_{\{d \in \Omega_{\text{odd}}\}}B_{d}^2\times \prod_{\{d \in \Omega_{\text{even}}\}}B_{d}^2$
\end{itemize}
where the subindex show the rational representation acting on the factor and $E_1, E_2$ are elliptic curves. \\
Moreover, the dimensions of the factors are given in general by 
\vskip6pt
\center{\rm
\begin{tabular}{|l|c|c|c|c|c|r|}\hline
\text{Actions}  & $E_1$ & $E_2 $& $B_{\{d \in \Omega_{\text{even}}\}}$ & $B_{\{d \in \Omega_{\text{odd}}\}}$\\\hline
\text{$\sigma_1$}   & 0 & 1 & 0 & $\varphi(\frac{2n}{d})$\\
\text{$\sigma_2$}& 1 & 0 & $\frac{1}{2}\varphi(\frac{2n}{d})$ & $\frac{1}{2}\varphi(\frac{2n}{d})$\\ \hline
\end{tabular}


\vskip12pt

\textsc{Table 3. Dimensions}
}

\end{thm}

\begin{proof}
We obtain Table 3 as  direct combination of Equations \eqref{dim-casodiedral-2} and Proposition \ref{propfix}. 
\end{proof}

\begin{rem}  
Notice that if we consider the  generating vectors \linebreak $(a^n,a^n,s,a^2s,a^2)$ and $(a^n,a^n,as,a^3s,a^2)$  corresponding to the same action $\sigma_1$ on $X$ (see Definition \ref{D:sigmas}), we obtain {\it essentially} the same decomposition of $JX$. The difference is the rational representation acting on the elliptic factor at the group algebra decomposition of $JX$. This is explained in Remark \ref{R:outer}.

On the other side, when we consider topologically non-equivalent actions, $\sigma_1$ and $\sigma_2$, we get decompositions of $JX$ with completely different behaviour: the factors have different geometry; dimension, structure and so on. 
\end{rem}

Using Lemma \ref{metodo} introduced in \cite{yo}, we have a criteria to isolate some products of the factors in the decomposition to write them as Jacobians (or Pryms) of intermediate converings. We recall here the results and notation introduced in subsection \ref{SS:factors}.

\begin{thm}\label{jacH}
Let $JX$ be as in Theorem \ref{dimensionesfactores}. Then $JX$ decomposes as 
\begin{itemize}
\item[(i)] for the action $\sigma_1$,  $JX\sim J(X/\langle a^2\rangle)\times J(X/\langle as\rangle)^2,$
\item[(ii)] for $\sigma_2$, $JX\sim J(X/\langle a^2\rangle)\times J(X/\langle a^n, s\rangle)^2 \times P(X_{\langle s\rangle}/X_{\langle a^n, s\rangle})^2$.
\end{itemize}
\end{thm}

\begin{proof}
Table 2 in Proposition \ref{propfix} contains the dimension of the fixed space of cyclic for all complex representations. We give the dimension for other subgroups in Table 4. 
 
\begin{center}
\begin{tabular}{|l|c|c|c|c|c|c|c|r|}\hline
Subgroups & $\chi_0$ & $\chi_1$ & $\chi_2$ & $\chi_3$ & $\psi_{\{d \in \Omega_{\text{even}}\}}$ & $\psi_{\{d \in \Omega_{\text{odd}}\}}$ \\ \hline
$\langle a^n, s\rangle$   & 1 & 0 & 0 & 0 & 1 & 0\\
$\langle a^{2}, s\rangle$   & 1 & 0 & 1& 0 & 0 & 0\\ 
$\langle a^{2}, as\rangle$   & 1 & 0 & 0& 1 & 0 & 0\\
\hline
\end{tabular}

\vskip12pt

\textsc{Table 4.} Dimension of fixed spaces.

\end{center}

By Theorem \ref{dimensionesfactores} we know that

\begin{itemize} 
\item[(i)] $\sigma_1$; $JX \sim E_2\times \Pi_{\{d \in \Omega_{\text{odd}}\}}B_{d}^2$
\item[(ii)] $\sigma_2$; $JX  \sim E_1\times \Pi_{\{d \in \Omega_{\text{even}}\}}B_{d}^2\times \Pi_{\{d \in \Omega_{\text{odd}}\}}B_{d}^2$
\end{itemize}

By Lemma \ref{metodo}, to write $E_2$ as a Jacobian of $X/H$ with $H\leq G$, we notice that in the table above joined with table in Proposition \ref{propfix} that $\dim \text{Fix}_H \chi_2=1$ and $\dim \text{Fix}_H \psi_{d\in \Omega_{\text{odd}}}=0$ when $H$ is a subgroup of $G$ in the class of subgroups represented by $\langle a^{2}\rangle, \langle a^{2}, s\rangle$. 

To write $E_1$ as a Jacobian of $X/H$ with $H\leq G$, we notice that $\dim \text{Fix}_H \chi_1=1$ and $\dim \text{Fix}_H \psi_{d}=0,$ for $d\in\Omega(2n)$ when $H$ is a subgroup of $G$ in the class of subgroups represented by $\langle a^{2}\rangle, \langle a\rangle$. 

To write the other piece in the decomposition corresponding to $\sigma_1$,  $\Pi_{\{d\in\Omega_{\text{odd}}\}}B_{d}^2$, we notice that we can not isolate every factor, but due to $\dim \text{Fix}_H \psi_{d}=1$ for all $d$ odd and $\dim\text{Fix}_H \chi_2=0$ when $H=\langle as\rangle$, using Equation (\ref{carocca-rodriguez}), and due to $\text{Ind}_H^G 1=\bigoplus_{\{d: d\in \Omega_{\text{odd}}\}} \cW_d$, we know that 
$$J_H\sim \Pi_{\{d\in\Omega_{\text{odd}}\}}B_{d}$$

To write $\Pi_{\{d \in \Omega_{\text{even}}\}}B_{d}$ in the decomposition for $\sigma_2$, we notice that $\dim \text{Fix}_H \psi_{d}=1$, $d\in \Omega_{\text{even}}$  and $\dim \text{Fix}_H \chi_1=0$ and $\dim \text{Fix}_H \psi_{d}=0$ for all $d$ odd when $H=\langle a^n, s\rangle$. Using Equation (\ref{carocca-rodriguez}) again, we know that $J_H\sim \Pi_{\{d \in \Omega_{\text{even}}\}}B_{d}.$

For the second item, we use \cite[Corollary 5.6]{cr} to prove that this product is a Prym variety of intermediate coverings. Then, since $\text{Ind}_H^G 1 - \text{Ind}_K^G 1= \chi_2 \bigoplus_{\{d: d\in \Omega_{\text{odd}}\}} \cW_d$ and $\dim E_2=0$ for $\sigma_2$, we get that 
$$P(X_H/X_K)\sim \Pi_{\{d\in\Omega_{\text{odd}}\}}B_{d},$$
where $H\leq K$ and $H=\langle s \rangle$, $K=\langle a^n, s\rangle$.

\end{proof}

\subsubsection{Particular case, $n=p$ a prime number}
When $n=p$ is a prime number, the rational irreducible representations of $D_{2p}$ are $\chi_0, \chi_1, \chi_2$ and $ \chi_3$ of degree one, and two $\cW_1, \cW_2$ of degree $p-1$, with associated complex representaion $\psi_1$ and $\psi_2$ respectively. The character fields of $\psi_1, \psi_2$ satisfy $[K_d:\QQ]=(p-1)/2,$ hence $\text{deg}(\mathcal{W}_d)=\dim_{\CC} \psi_d [K_d:\QQ]=p-1$ for $d\in \{1,2\}$.

Then the group algebra decomposition of $JX$ is given by
$$JX\sim E_0\times E_1\times E_2\times E_3\times B_1^2\times B_2^2,$$ 
where we know that $E_0=J(X/D_{2p})$ has dimension $0$.

\begin{cor}\label{dimensionesfactores_p}
If $X$ is a Riemann surface with action of the group $D_{2p}$ and signature $(0;2,2,2,2,p)$ then $JX$ has always one elliptic factor in its group algebra decomposition. 
The dimensions of the factors are given in general by 
\vskip6pt
\center{\rm
\begin{tabular}{|l|c|c|c|c|c|c|c|r|}\hline
\text{Actions} & $E_0$ & $E_1$ & $E_2 $& $E_3$ & $B_1$ & $B_2$ \\ \hline
\text{$\sigma_1$}   &  0 & 0 & 1 & 0 & $p-1$& 0\\
\text{$\sigma_2$}&  0 & 1 & 0 & 0 & $\frac{p-1}{2}$ & $\frac{p-1}{2}$\\ \hline
\end{tabular}

\vskip12pt

\textsc{Tabla 5.} Dimensions.
}

\vskip8pt
\begin{itemize}
\hspace{-1.6cm}In particular, if $p=3$ we have that, for the action\\

\item[(i)] $\sigma_1$; the Jacobian $JX\sim E_2\times B_1^2$, where $E_2$ is an elliptic curve and $B_1$ an abelian surface.
\item[(ii)] $\sigma_2$; the Jacobian $JX\sim E_1\times B_1^2\times B_2^2$, where $E_1,B_1$ and $B_2$ are elliptic curves. Hence in this case $JX$ is completely decomposable.
\end{itemize}
\end{cor}
\begin{proof}
We obtain the table above directly from Equation \ref{dim-casodiedral-2} and Proposition \ref{propfix}. The case $p=3$ is a particular case of this.
\end{proof}

As a direct consequence of Theorem \ref{jacH}, we obtain the following corollary.
\begin{cor}\label{C:jacH}
If $X$ is a Riemann surface with action of the group $D_{2p}$ and signature $(0;2,2,2,2,p)$ then the factors of the decomposition of $JX$ in Theorem \ref{dimensionesfactores} are Jacobian of intermediate coverings or a Prym variety in one case. Indeed we obtain that the decomposition is 
\begin{itemize}
\item[(i)] for the action  $\sigma_1$, $JX\sim J(X/\langle a^2\rangle)\times J(X/\langle s\rangle)^2$.
\item[(ii)] for $\sigma_2$, $JX\sim J(X/\langle a^2\rangle)\times J(X/\langle a^p, s\rangle)^2\times P(X_{\langle s\rangle}/X_{\langle a^p, s\rangle})^2$.
\end{itemize}
\end{cor}

\subsection{Case $D_{2n}$, $n$ even number}
Consider the group $D_{2n}$ and its action as in Sections \ref{S:strata} and \ref{S:equations}. Recall that $D_{2n}$ has four complex irreducible representations of degree one, and $n-1$ of degree two. All of them with Schur index equal to 1. 


The rational irreducible representations of $D_{2n}$ are $\chi_0,\chi_1, \chi_2,\chi_3$ of degree one, and 
$$\cW_{d}:=\cW_{\psi_d}=\bigoplus_{\sigma\in \text{Gal}(K_{\psi_{d}}/\QQ)} \psi_{d}^{\sigma},$$
for all $d\in \Omega(2n)$. 

Then the group algebra decomposition of $JX$ is given by
$$JX\sim E_0\times E_1\times E_2\times E_3\times \Pi_{d\in \Omega(2n)} B_d^2,$$ 
where $E_i$ are subvarieties of $JX$ associated to the characters $\chi_i$,  and we know that $B_0=J(X/D_{2n})$ is of dimension $0$. 

In this section we do not include the proofs of the theorems that follow because they are analogous to the proofs corresponding to the action $\sigma_2$ in Section \ref{section-odd}.   







\begin{thm}\label{dimensionesfactores-par}
Let $n$ be an even number.
If $X$ is a Riemann surface having the unique action of the group $D_{2n}$ with signature $(0;2,2,2,2,n)$, then $JX$ has always one elliptic factor in its group algebra decomposition. In fact the decomposition is 

$$JX  \sim E_1\times \Pi_{\{d \in \Omega(2n)\}}B_{d}^2$$
where the subindex show the rational representation acting on the factor. $E_1$ is an elliptic curve and $\dim B_d=\frac{1}{2}\varphi(\frac{2n}{d})$. 



\end{thm}

\vspace{0.2cm}
As we show in Theorem \ref{jacH}, using Lemma \ref{metodo}, we have a criterium to write some products of the factors as Jacobian of intermediate converings. 

\begin{thm}\label{jacH-par}
Let $n$ be an even number.
If $X$ is a Riemann surface having the unique action of the group $D_{2n}$ with signature $(0;2,2,2,2,n)$. Then $JX$ decomposes as

$$JX\sim J(X/\langle a^2\rangle)\times J(X/\langle a^n, s\rangle)^2\times P(X_{\langle s\rangle}/X_{\langle a^n, s\rangle})^2.$$

\end{thm}



\section{Shimura domains and Jacobians.}\label{S:shimura}

In this section we compute the dimension of the Shimura domains associated to the Jacobian varieties corresponding to the curves with action of $G=D_{2n}$ with signature $(0;2,2,2,2,n)$. We follow the ideas in \cite[Section 3]{wolfart}. Let $X$ be a Riemann surface of genus $g$ with the action of a finite group $G$ defined by a generating vector. In our case $g=2n-1$, $G=D_{2n}$, and we have two non-topologically equivalent actions defined by $\sigma_1$ and $\sigma_2$ in the case $n$ is odd, and one action $\sigma_2$ (up to topological equivalence) if $n$ is even. We show that the Shimura domains have different dimension depending on the action. This is, for $n$ odd the strata $\mathcal{M}(g,G,\theta_1)$ goes via the Jacobi map to a submanifold $\mathcal{S}_1$ of the moduli space $\mathcal{A}_g$ of principally polarized abelian varieties of dimension $N_{n,\sigma_1}$, and for all $n$ the strata $\mathcal{M}(g,G,\theta_2)$ goes to a submanifold $\mathcal{S}_2$ of dimension $N_{n,\sigma_2}$. For $n$ even, the situation is as in $\mathcal{M}(g,G,\theta_2)$ for $n$ odd.

Let us develop some background first. Fix a symplectic basis of the homology of $X$, this determines a fixed Riemann matrix $Z\in \mathbb{H}_g$ in the Siegel upper half space of complex $g\times g$ symmetric matrices with positive definite imaginary part, of the Jacobian $JX$ of $X$. This choice also determines a symplectic representation of $L:=\End_0 (JX)=\End (JX)\otimes_{\mathbb{Z}} \mathbb{Q}$. Every automorphism of $X$ induces a unique automorphism of the corresponding Jacobian, hence $G$ can be considered as a subgroup of the  polarization preserving automorphism of $JX$. Therefore $G$ is isomorphic to the subgroup

$$\Sigma:=\{\gamma \in \Sp_{2g}(\mathbb{Z}) : \gamma*Z=Z\},$$

where the action of $\gamma =\left(\begin{array}{c c} A & B\\ C & D\\ \end{array}\right)$ is 

$$\left(\begin{array}{c c} A & B\\ C & D\\ \end{array}\right)*Z=(A+ZC)^{-1}(B+ZD).$$

A change of basis induces a different (but equivalent) choice of $Z$, a conjugated subgroup $\Sigma$ and a different rational representation of $L$. Fixing $Z$ it is obtain a submanifold 

$$S_G:=\{W\in \mathbb{H}_g: \gamma*W=W \text{ for all } \gamma \in \Sigma\}.$$

$S_G$ contains a complex submanifold $\mathbb{H}(L)$ of $\mathbb{H}_g$ parametrizing a {\em Shimura family} $\mathcal{S}$ of principally polarized abelian varieties containing $L$ in their endomorphism algebras. $\mathbb{H}(L)$ is called \cite[Section 3]{wolfart} the {\em Shimura domain} for $Z$.

According to \cite[Lemma 3.8]{paola}, the dimension of $\mathbb{H}(L)$ corresponds to 

$$N=\dim\; (S^2(H^{1,0}(S,\mathbb{C})))^G,$$ 
where $H^{1,0}(S,\mathbb{C})$ is the complex vector space of holomorphic forms. The action of $G$ on $X$ induces a (analytic) representation $\rho_a$ in this vector space. Let $S^2(H^{1,0}(S,\mathbb{C}))$ be the representation of $G$ on the symmetric power of $H^{1,0}(S,\mathbb{C})$, and $(S^2(H^{1,0}(S,\mathbb{C})))^G$ the subspace fixed by $G$. The dimension of this last subspace can be computed in terms of the character of $\rho_a$ because, by Frobenius reciprocity, it corresponds to the character product of the trivial representation of $G$ with $S^2(H^{1,0}(S,\mathbb{C}))$.

Using Serre's formula \cite[Section 1.2]{serre}

\begin{equation}\label{eq:N}
 N=\frac{1}{2|G|}\sum_{g\in G}\left(\chi_{\rho}(g)^2+\chi_{\rho}(g^2)\right),
\end{equation}

where $\chi_\rho$ is the character of $\rho_a$. 
Let us write
\[
 \chi_\rho=\sum_{\chi\in\text{Irr}(G)}\mu_\chi\chi.
\]
In general, the coefficients $\mu_\chi$ can be computed using a classical result due to Chevalley and Weil \cite{chevw}. In our case, since we know the isotypical decomposition of $JX$, we know the irreducible representations in $\rho_a$.

Let us compute detailed these dimensions for the case $n$ a prime number. The general case is similar, but the technical details are harder to be written.

\begin{prop}\label{P:Ns}
Let $p$ be a prime number and $D_{2p}=\langle a,b | a^{2p}=s^2=(as)^2=1\rangle$. Let $\sigma_1$ and $\sigma_2$ be two generating vectors representing the two non-topologically equivalent actions of $D_{2p}$ with signature $(0;2,2,2,2,p)$. For $i=1,2$, denote by $N_{p,\sigma_i}$ the dimension of the Shimura domain corresponding to the action determined by $\sigma_i$. Then

$$N_{p,\sigma_1}=\frac{3p-1}{2}\;\; \text{ and } \; N_{p,\sigma_2}=p.$$

\end{prop}

\begin{proof}
Let us denote by $\rho_{\sigma_i}$ the analytic representation corresponding to the action determined by $\sigma_i$. From Theorem \ref{dimensionesfactores} we have that

$$\rho_{\sigma_1}\equiv \chi_2\oplus 2W_1 \text{ and } \rho_{\sigma_2}\equiv \chi_1\oplus W_1\oplus W_2,$$

The character table for these representations is given in Table 7.

\vskip12pt

\begin{center}
\begin{tabular}{c|c|c|c|c|c|c|c|c}
Rep. & 1& $a$ &$a^2$ & $\cdots$ & $a^{p-1}$ & $a^p$ & $s$ & $as$\\ \hline
$\sharp$ & $1$ & 2& 2& $\cdots$ & $2$ & $1$ & $p$ & $p$ \\ \hline \hline
$\chi_1$ & 1& 1& 1& $\cdots$ & $1$ & $1$ & $-1$ & $-1$ \\ \hline
$\chi_2$ & 1& -1& 1& $\cdots$ & $1$ & $-1$ & $1$ & $-1$ \\ \hline
$W_1$ & $p-1$& 1& -1& $\cdots$ & $1$ & $-(p-1)$ & $0$ & $0$ \\ \hline
$W_2$ & $p-1$& -1& -1& $\cdots$ & $-1$ & $(p-1)$ & $0$ & $0$ \\ \hline

\end{tabular}

\vskip12pt

\textsc{Table 7.} Characters to compute the analytic character, prime case.
\end{center}

Using Equation \eqref{eq:N} to compute $N$ in our context, we have

$$N_{p,\sigma_1}=\frac{1}{8p}\sum_{g\in D_{2p}}\left((\chi_2(g)+2\chi_{W_1}(g))^2+(\chi_2(g^2)+2\chi_{W_1}(g^2))\right).$$
The result follows from replacing the values from Table 7 in the term $\left((\chi_2(g)+2\chi_{W_1}(g))^2+(\chi_2(g^2)+2\chi_{W_1}(g^2))\right)$ for each element $g\in G$:

\begin{enumerate}
\item $g=1$, $((1+2(p-1))^2+(1+2(p-1)),$
\item $g=a^{2j-1}$, $(-1+2)^2+(1-2)=0$, 
\item $g=a^{2j}$, $(1-2)^2+(1-2)=0$, 
\item $g=a^p$, $(-1-2(p-1))^2+(1+2(p-1)),$
\item $g=a^{2j-1}s$, $(1)^2+(1+2(p-1))$, 
\item $g=a^{2j}s$, $(-1)^2+(1+2(p-1))$. 
\end{enumerate}

For the other action we have

$$N_{p,\sigma_2}=\frac{1}{8p}\sum_{g\in D_{2p}}\left((\chi_2(g)+\chi_{W_1}(g)+\chi_{W_2}(g))^2+(\chi_2(g^2)+\chi_{W_1}(g^2)+\chi_{W_2}(g^2))\right).$$
The result follows in an analogous way as before. Here the terms are

\begin{enumerate}
\item $g=1$, $((1+(p-1)+(p-1))^2+(1+(p-1)+(p-1)),$
\item $g=a^{2j-1}$, $(1+1-1)^2+(1-1-1)=0$, 
\item $g=a^{2j}$, $(1-1-1)^2+(1-1-1)=0$, 
\item $g=a^p$, $(1-(p-1)+(p-1))^2+(1+2(p-1)),$
\item $g=a^{2j-1}s$, $p((-1)^2+(1+2(p-1)))$, 
\item $g=a^{2j}s$, $p((-1)^2+(1+2(p-1)))$. 
\end{enumerate}

\end{proof}

\begin{rem}
In \cite[Thm. 3.9]{paola} there is a criterion allowing to prove whether a family of Jacobians contains infinitely many elements with complex multiplication. It is as follows, let  $G$ be a finite group acting on a curve of genus $g$ with signature $m=[0;m_1,\dots,m_r]$, and generating vector $\sigma=(c_1,\dots, c_r)$. For a fixed pair $(m,\sigma)$, by moving the branch points of the covering in $\mathbb{P}^1$ one obtains an $(r-\!3)$-dimensional family of such coverings, and a corresponding family of Jacobians $\mathcal{J}(G,m,\sigma)$ of the same dimension. Denote by $Z(G,m,\sigma)$ the closure in the moduli space of principally polarized abelian varieties $\mathcal{A}_g$ of $\mathcal{J}(G,m,\sigma)$, this is the image in $\mathcal{A}_g$ of the connected component of $\mathbb{H}(L)$ containing $\mathcal{J}(G,m,\sigma)$. Let $N=\dim (S^2(H^{1,0}(S,\mathbb{C})))^G$ as before. If $N$ is $(r-\!3)$, then $Z(G,m,\sigma)$ contains a dense set of Jacobians of CM-type. Notice that our family has the action with $5$ branch points, hence it is of dimension $2$, and for $p=2$ we obtain $N_2=2$, therefore this family gives infinitely many Jacobian varieties of CM type in dimension $g=3$. This is the same family $(32)$ in \cite[Table 2]{paola}.

\end{rem}

For the general case we have the following result.

\begin{thm}
With the notation in Proposition \ref{P:Ns}. Let $n$ be an integer. If $n$ is odd, let $\sigma_1$ and $\sigma_2$ be the two non-topologically equivalent actions of $D_{2n}$ with signature $(0;2,2,2,2,n)$. For $i=1,2$, denote by $N_{n,\sigma_i}$ the dimension of the Shimura domain corresponding to the action determined by $\sigma_i$. Then

$$N_{n,\sigma_1}=\frac{3n-1}{2}\;\; \text{ and } \; N_{n,\sigma_2}=n.$$

If $n$ is even, let $\sigma_2$ be the unique action of $D_{2n}$ with signature \linebreak $(0;2,2,2,2,n)$, and denote by $N_{n,\sigma_2}$ the dimension of the corresponding Shimura domain. Then $N_{n,\sigma_2}=n$. 

\end{thm}
\begin{proof}
As said, the proof follows the same strategy as in Proposition \ref{P:Ns}, but one needs to describe the character of the analytic representation using the decompositions for each case. Let us first review the case $n$ an odd number. In Theorem \ref{dimensionesfactores} is described the decomposition of the corresponding Jacobian variety, hence the analytic character corresponds to the following for each action. Let us denote them as $\rho_{\sigma_1}$ and $\rho_{\sigma_2}$ respectively.

Since $n$ is odd, the set $\Omega(2n)$ decomposes as $\Omega_{even}\cup\Omega_{odd}$, as in \eqref{eq:omega_odd}. Therefore for each action $\sigma_i$ we have the following decomposition of the analytic representation

\begin{center}
\begin{tabular}{l|c}
Action& $\rho_{\sigma_i}$\\ \hline
$\sigma_1$ & $\displaystyle{\chi_2 \bigoplus_{j\in \Omega_{odd}} 2\psi_j}$\\ \hline
$\sigma_2$ & $\displaystyle{\chi_1 \bigoplus_{j\in \Omega_{even}} \psi_j \bigoplus_{j\in \Omega_{odd}} \psi_j}$\\ \hline
\end{tabular}

\vskip12pt

\textsc{Table 8.} Decomposition of the analytic representation.
\end{center}

Denote by $\widetilde{\mathcal{W}_1}=\bigoplus_{j\in \Omega_{odd}} \psi_j$ and $\widetilde{\mathcal{W}_2}=\bigoplus_{j\in \Omega_{even}} \psi_j$. The character table is given in Table 9. 

For $n$ even, we have just one action $\sigma_2$ up to topological equivalence. The result follows from computing the dimensions $N_{n,\sigma_1}$ and $N_{n,\sigma_2}$ using Equation \eqref{eq:N}.

\begin{center}
\begin{tabular}{c|c|c|c|c|c|c|c|c}
Rep. & 1& $a$ &$a^2$ & $\cdots$ & $a^{n-1}$ & $a^n$ & $s$ & $as$\\ \hline
$\sharp$ & $1$ & 2& 2& $\cdots$ & $2$ & $1$ & $n$ & $n$ \\ \hline \hline
$\chi_1$ & 1& 1& 1& $\cdots$ & $1$ & $1$ & $-1$ & $-1$ \\ \hline
$\chi_2$ & 1& -1& $(-1)^2$ & $\cdots$ & $(-1)^{n-1}$ & $(-1)^n$ & $1$ & $-1$ \\ \hline
$\widetilde{\mathcal{W}_1}$ & $n-1$& 1& -1& $\cdots$ & $1$ & $-(n-1)$ & $0$ & $0$ \\ \hline
$\widetilde{\mathcal{W}_2}$ & $n-1$& -1& -1& $\cdots$ & $-1$ & $(n-1)$ & $0$ & $0$ \\ \hline

\end{tabular}

\vskip12pt

\textsc{Table 9.} Characters to compute the analytic character.

\end{center}

\end{proof}

\begin{rem}
In\cite{lrro} a family $\mathcal{F}$ of principally polarized abelian varieties with action of $D_{2p}$ is studied, for $p$ an odd prime. This family is different from ours because of several reasons; first of all, $\mathcal{F}$ contains no Jacobians \cite[Proposition 6.9]{lrro}, the elements in $\mathcal{F}$ have dimension $2p$, and the generic element decomposes as $E_0\times  E_3\times B_1^2 \times B_2^2$, where the subvarieties are associated to the rational representations of $D_{2p}$ as in Corollary \ref{dimensionesfactores_p} and $E_i$'s are elliptic curves and $B_i$'s are of dimension $\frac{p-1}{2}$. Finally the Shimura domain for $\mathcal{F}$ has dimension $p+1-$ and ours are of dimension $p$ or $\frac{3p-1}{2}$.
\end{rem}



\begin{thebibliography}{99999}

\bibitem{accola1} R.D.M. Accola,
\textit{On cyclic trigonal Riemann surfaces I.}
Trans. Amer. Math. Soc. \textbf{283} (1984), no. 2, 423--449.

\bibitem{brr} A. Behn, R. E. Rodr\'iguez, A.M. Rojas, 
\textit{Adapted Hyperbolic Poligons and Symplectic Representations for
group actions on Riemann surfaces}. 
J. Pure Appl. Alg. \textbf{217} (2013), 409--426.

\bibitem{bl}
Ch.~Birkenhake, H.~Lange,
\textit{Complex Abelian Varieties.} $2^{nd}$ edition,
Grundl. Math. Wiss. \textbf{302}, Springer-Verlag, Heidelberg, 2004.

\bibitem{magma} W. Bosma, J. Cannon, and C. Playoust,
\textit{The Magma algebra system. I. The user language.}
J. Symbolic Comput. \textbf{24}  (1997), 235--265. http://magma.maths.usyd.edu.au

\bibitem{bro} A. Broughton. {\it The equisymmetric stratification of the moduli space and the Krull dimension of mapping class groups}. Topology Appl. \textbf{37}  (1990), 101--113.

\bibitem{crro} A. Carocca, S. Recillas, and  R. E. Rodr\'iguez, 
\textit{Dihedral groups acting on Jacobians.}
Contemp. Math., \textbf{311}  (2002), 41--77. 

\bibitem{cr} A. Carocca, and R. E. Rodr\'iguez,
\textit{Jacobians with group actions and rational idempotents.}
J. Algebra \textbf{306} (2006),  322--343. 

\bibitem{chevw}
    C. Chevalley, A. Weil,
    \textit{\"Uber das Verhalten der Integrale erster Gattung bei Automorphismen des Funktionenk\"orpers},
	Hamb. Abh.  \textbf{10} (1934), 358--361. 

\bibitem{curtis} C.W. Curtis, I. Rainer, 
\textit{Representation Theory of Finite Groups and Associative Algebras}.
Wiley, New York, 1988

\bibitem{serre} T.Ekedhal, and J.-P.Serre,
\textit{Examples des corbes alg\'ebriques \`a jacobienne compl\`etement d\'ecomposable}.
C.R. Acad. Paris S\'er. I Math. \textbf{317} (1993), 509--513.

\bibitem{fk}
     H. Farkas, I. Kra,
     \textit{Riemann Surfaces.}
     Graduate Texts in Mathematics \textbf{72}, Springer-Verlag, New York, 1996.

\bibitem{paola} 
{P. Frediani, A. Ghigi, \and M. Penegini}, 
\textit{Shimura varieties in the {T}orelli locus via {G}alois coverings},
 Int. Math. Res. Not. \textbf{20}  (2015), 10595--10623.

\bibitem{harvey} 
{W. Harvey}, 
\textit{On branch loci in Teichm\"uller space},
 Trans. Amer. Math. Soc. \textbf{153}  (1971), 387--399.

\bibitem{james-liebeck}
G. James and M.Liebeck,
\textit{Representations and Characters of Groups.} 
Cambridge Mathematical Textbooks, Cambridge University Press, Cambridge, 1995.

\bibitem{yo} L. Jim\'enez,
 \textit{On the kernel of the group algebra decomposition of a Jacobian variety}. 
Rev. Real Academia de Ciencias, Fisicas y Naturales. Serie A matem\'aticas. RACSAM \textbf{110} (2016), 185--199.

\bibitem{kr} E. Kani, M. Rosen,
 \textit{Idempotent relations and factors of Jacobians}. Math. Ann. \textbf{284} (1989), 307--327.

\bibitem{lro} H. Lange and A.M. Rojas,
\textit{Polarizations of isotypical components of Jacobians with group action}. 
Arch. der Math. \textbf{98} (2012), 513--526.

\bibitem{lrro} H. Lange, R. E. Rodr\'iguez and A.M. Rojas,
\textit{Polarizations on abelian subvarieties of principally polarized abelian varieties with dihedral group action}. 
Math. Z. \textbf{276} (2014), 397--420.

\bibitem{l-r}
H. Lange and S. Recillas,
\textit{Abelian varieties with group actions}.
Journ. Reine Angew. Mathematik, \textbf{575} (2004), 135--155.

\bibitem{macbeath}
A. M. Macbeath,
\textit{Action of automorphisms of a compact Riemann surface on the first homology group}. 
Bull. London Math. Soc. \textbf{5} (1973), 103--108.

\bibitem{macbeath-singerman}
A. M. Macbeath and D. Singerman,
\textit{Spaces of subgroups and Teichm\"uller space}.
Proc. London Math. Soc., \textbf{31}(2) (1975), 211--256.


\bibitem{magaard}
K. Magaard, T. Shaska, S. Shpectorov and H. V\"olklein, 
\textit{The locus of curves with prescribed automorphism group}.
Comm. Arithm. Fund. Groups (Kyoto),
Surikaisekikenkyusho Kokyuroku No. \textbf{1267} (2002), 112--141.

\bibitem{nag} S. Nag,
\textit{The Complex Analytic Theory of Teich,\"uller Spaces.} Wiley-Interscience Publications, New York, 1998.


\bibitem{naka} R. Nakayima, 
\textit{On splitting of certain Jacobian varieties}.
J. Math. Kyoto Univ. (JMKYAZ) \textbf{47-2} (2007), 391--415.

\bibitem{paulhus} J. Paulhus, 
\textit{Decomposing Jacobians of curves with extra automorphisms}.
Acta Arith. \textbf{132} (2008), 231--244.

\bibitem{paulhus2} J. Paulhus, 
\textit{Elliptic factors in Jacobians of hyperelliptic curves with certain automorphism groups.}
Proceedings of the Tenth Algorithmic Number Theory Symposium (2013), 487--505

\bibitem{jaime} 
{J. Pinto}, 
\textit{Ecuaciones para Superficies de Riemann correspondientes a cubrimientos c\'iclicos primos},
{\em Master Thesis}, Universidad de Chile,(2013).

\bibitem{sevin} S. Recillas,
\textit{Jacobians of curves with $g_4^1$'s are the Prym's of trigonal curves.}
Bol. Soc. Mat. Mexicana (2) \textbf{19} (1974), no. 1, 9--13. 

\bibitem{reci-ro} S. Recillas, R. Rodr\'iguez, 
	\emph{Jacobians and representations of $S_3$}, 
Aportaciones Mat. Investig. \textbf{13}, Soc. Mat. Mexicana, M\'exico, 1998.

\bibitem{rubi} R. Rodr\'iguez, 
	\emph{Abelian Varieties and Group Actions}, 
Contemp. Math.  \textbf{629} (2014), 299--314.

\bibitem{yoibero}
A. M.   Rojas,
    \textit{Group actions on Jacobian varieties}.
    Rev. Mat. Iber. \textbf{23} (2007), 397--420.

\bibitem{serre1}
    J-P. Serre, 
    \textit{Linear Representations of Finite Groups.}
    Grad. Texts in Math. \textbf{42}, Springer-Verlag, Heidelberg, 1996.

\bibitem{singerman} D. Singerman. {\it Subgroups of Fuschian groups and finite permutation groups.} Bull. London Math. Soc. 
\textbf{2}  (1970), 319--323.

\bibitem{wolfart}
J. Wolfart,
\textit{Regular dessins, Endomorphisms of Jacobians ans Transcendece},
{\em A Panorama of Number Theory or the View from Baker's Garden,} ed. G. W\"ustholz, Cambridge University Press,  Cambridge (2002), 107-120.

\bibitem{w1} 
{A. Wooton}, 
\textit{Defining equations for cyclic prime covers of the Riemann Sphere}, 
 Israel Journal of Mathematics \textbf{157} (2007), 103-122.
\end{thebibliography}
\end{document}